\documentclass[a4paper,11pt]{article}
\usepackage{bbm}
\usepackage{amsfonts}
\usepackage{amsthm,amsmath}
\usepackage[integrals]{wasysym}
\usepackage{mathrsfs}
\usepackage{amsmath,multicol,amsopn}
\usepackage{latexsym,amssymb}
\usepackage{amsbsy, amstext}
\usepackage{graphicx}    
\usepackage{indentfirst} 
\usepackage{graphicx}
\usepackage{subfigure}
\usepackage{color}
\def\be{\begin{array}}
\def\en{\end{array}}

\numberwithin{equation}{section}

\newtheorem{theorem}{Theorem}[section] 
\newtheorem{definition}{Definition}[section]
\newtheorem{coro}{Corollary}[section]
\newtheorem{lemma}{Lemma}[section]
\newtheorem{remark}{Remark}[section]
\newtheorem{proposition}{Proposition}[section]

\newcommand{\R}{{\mathbb R}}
\newcommand{\pozhehao}{\kern0.3ex\rule[0.8ex]{1.5em}{0.095ex}\kern0.3ex}

\def\x#1{{\rm (\ref{#1})}}

\begin{document}
\title{\LARGE\bf{Global Regularity and instability for the incompressible non-viscous Oldroyd-B model}
\footnote{}}
\date{}
\author{Zhi Chen$^{1}$\thanks{{\small E-mail: zhichenmath@csu.edu.cn (Z. Chen)}},
 ~~Weikui Ye$^{2}$\thanks{{\small E-mail: 904817751@qq.com (W. Ye)}},~~
 Zhaoyang Yin $^{3}$\thanks{{\small E-mail:  mcsyzy@mail.sysu.edu.cn (Z. Yin)}},\\
{\small 1. School of Mathematics and Statistics, Central South University,}\\
{\small Changsha, 410083 Hunan, P. R. China}\\
{\small 2. $\mbox{Institute}$ of Applied Physics and Computational Mathematics,}\\
{\small P.O. Box 8009, Beijing 100088, P. R. China}\\
{\small 3.  Department of Mathematics, Sun Yat-sen University,}\\
{\small Guangzhou, 510275 Guangdong, P. R. China}\\
}\maketitle

\begin{center}
\begin{minipage}{15.5cm}

{\bf Abstract.} In this paper, we consider the 2-dimensional non-viscous Oldroyd-B model.
In the case of the ratio equal 1~($\alpha=0$), it is a difficult case since the velocity field $u(t,x)$ is no longer decay. Fortunately, by {observing the exponential decay} of the stress tensor $\tau(t,x)$, we succeeded in proving the global existence for this system with some large initial data.
Moreover, we give an unsteady result: when the ratio is close to 1~($a\rightarrow 0$), the system is not steady for large time. This implies an interesting physical phenomenon that the term $a\mathbb{D}u$ is a bridge between the transformation of kinetic energy $u$ and elastic
potential energy $\tau$, but this process is transient for large time, which leads the instability.


\end{minipage}
\end{center}
\vskip 6mm

\section{Introduction and main results}
 \setcounter{equation}{0}
\par
In this paper, we study the incompressible Oldroyd-B model of the non-Newtonian fluid in $\R^+\times\R^d$
\begin{equation}\label{gs}
\begin{cases}
\partial_tu+(u\cdot\nabla)u-\nu\Delta u+\nabla p=\mu_1\mathrm{div}(\tau),\\
\partial_t\tau+(u\cdot\nabla)\tau-\eta\Delta\tau+\mu_2\tau+\mathrm{Q}(\nabla u,\tau)=\alpha\mathbb{D}u,\\
\mathrm{div} u=0,\\
u(x,0)=u_0(x),~~~\tau(0,x)=\tau_0(x),
\end{cases}
\end{equation}
where $u$ denotes the velocity, $\tau=\tau_{i,j}$ is the non-Newtonian part of the stress tensor($\tau$ is a $d\times d$ symmetric matrix here)
and $p$ is a scalar pressure of fluid.
 $\mathrm{D}(u)$ is the symmetric part of the velocity gradient,
 \begin{equation*}
   \mathrm{D}(u)=\frac{1}{2}(\nabla u+(\nabla u)^T).
 \end{equation*}
  The $\mathrm{Q}$ above is a given bilinear form:
\begin{equation*}
 \mathrm{Q}(\tau,\nabla u)=\tau\Omega(u)-\Omega(u)\tau+b(\mathrm{D}(u)\tau+\tau \mathrm{D}(u)),
\end{equation*}
where $b$ is a parameter in $[-1,1],$ $\Omega(u)$ is the skew-symmetric part of $\nabla u,$ i.e.
\begin{equation*}
  \Omega(u)=\frac{1}{2}(\nabla u-(\nabla u)^T).
\end{equation*}
The parameters $\nu,\eta,\mu_1,\alpha$ are non-negative and they are specific to the characteristic of
the considered material. In \cite{LM}, $\nu,\mu_2$ and $b$ correspond respectively to $\theta/Re, 1/We$ and $2(1-\theta)/(WeRe),$
where $Re$ is the Reynolds number, $\theta$ is the ratio between the relaxation and retardation times and $We$ is
the Weissenberg number.
In this paper, we will investigate  the case $\alpha=0$. Following \cite{CM,MR1633055},we claim that the limit model $\alpha=0$ occurs when
$\theta$, the ratio between the so called relaxation and retardation times, is converging to 1.
We specifically pointed out some different Oldroyd models with infinite Weissenberg number (and so $\alpha,\mu_2=0$) have been investigated in \cite{HY}.
\par
The Oldroyd-B model describes the motion of some viscoelastic flows. Formulations about viscoelastic flows of Oldroyd-B type are first
established by Oldroyd in \cite{OJ}. For more detailed physical background and derivations about this model, we refer the readers to \cite{BCF,PM,139,OJ}.
\par
The well-posedness of the system \x{gs} has been {extensively and continuously studied.} Guillop\'{e} and Saut \cite{GS1,GS2} obtained the local solutions
with large initial data and showed that these solutions are global when the coupling and the initial data are small enough.
In the corotational case, i.e. $b=0,$ Lions and Masmoudi established the existence of global weak solution in \cite{LM}.

  Many works have been devoted to find the global solution in the case
of small data after the work of F. Lin et al.\cite{MR2165379}. Chemin and Masmoudi \cite{CM}
first obtained the local solutions and global small solutions  {in the critical Besov spaces} when $\nu>0, \mu_1>0, \alpha>0,$ and $\eta=0.$ { They get the global small solutions when the  initial and coupling parametra is small, i.e.($\mu_1\alpha\leq c\mu_2\nu$).
The condition $\mu_1\alpha\leq c\mu_2\nu$ means that coupling effect between the two equation is less important than the viscosity.}
{Inspired} by the work \cite{CW,CMZ}, Zi, Fang and Zhang improved their results in the critical $L^p$ framework for the case
of non-small coupling parameters in \cite{ZFZ}.
\par
Some recent results dealt with the case when there is only kinematic dissipation or stress
tensor dissipation,i.e. the systems do not have damping term. Zhu \cite{145} got small global smooth solutions of the $\mathrm{3D}$ Oldroyd-B model with $\eta=0, \mu_2=0$ by observing the linearization of the system satisfies the damped wave equation. Inspired by the work of Zhu \cite{145} and Danchin in \cite{DR}, Chen and Hao \cite{CH2} extended this small data global solution in Sobolev spaces to the critical Besov spaces.    Wu and Zhao \cite{WZ} and Zhai \cite{ZYZ2}  established the small data global well-posedness in critical Besov spaces for fractional dissipation of velocity respectively. Moreover, Zhai \cite{zhai} constructs global solutions for a class of highly oscillating initial velocities by observing the special structure of the system.
\par
When the damping mechanism and the Laplacian dissipation
for $\tau$ exist,  Elgindi and Rousset \cite{ER} established a global large solution in a certain sense by building a new quantity {to avoid singular operators.}
Later, Liu and Elgindi extend these results to the 3D in \cite{135}. Constantin, Wu, Zhao and Zhu \cite{wjh8} established the small data global solution to the n-dimensional(n=2,3) in the case of no damping mechanism and  general {tensor} dissipation. Recently, Wu and Zhao \cite{WZ2} investigated the small data global well-posedness in Besov spaces for fractional {tensor} dissipation. Very recently, Chen, Liu, Qin and Ye \cite{CZL}
improved their results in the critical $L^p$ framework by using Lagrangian coordinates in the high frequency.
\par
It is worth pointing that the exist of $\mathrm{D}u$ is key for the proof of the above results. A nature question is that
Whether the system has a global solution when the {$\alpha=0$} in \eqref{gs}? Since the $\mathrm{D}u$ play a important role by
the equality
\begin{equation*}
  \int_{\R^2} \mathrm{div}\tau\cdot udx+\int_{\R^2} Du:\tau=0.
\end{equation*}
Moreover, the term $\mathrm{D}u$ ensures the damping term of $\Gamma$ in paper \cite{ER}, in the work of \cite{wjh8},
the $\mathrm{D}u$ guarantee the  linearized system satisfy the damped wave structure. In this paper, we will give a positive answer
by the Theorem 2. As far as {we} know, this is the first study the case of $\alpha=0$ in \eqref{gs} with any $b\in [-1,1]$ in $Q(\nabla u, \tau)$.
\par
When $\alpha=0,$ the system will become interesting which do not enjoy the special structure similar to \cite{wjh8}.
 In this paper, We can get a  well-posedness  result in the case of $\alpha=0.$
We also study how the global strong solutions of \eqref{gs} behave when the ratio $\theta$ trends to 1, (i.e. $a\rightarrow 0$). Without lose of generality, we let the  $\nu=0, \mu_1=1, \eta=1$ and $\mu_2=1$, then \eqref{gs} becomes:
\begin{equation}\label{gs2222}
\begin{cases}
\partial_tu+(u\cdot\nabla)u+\nabla p=\mathrm{div}(\tau),\\
\partial_t\tau+(u\cdot\nabla)\tau-\Delta\tau+\tau+\mathrm{Q}(\nabla u,\tau)=a\mathbb{D}u,\\
\mathrm{div} u=0,\\
u(x,0)=u_0(x),~~~\tau(0,x)=\tau_0(x),
\end{cases}
\end{equation}
In \cite{ER}, Elgindi and Rousset established a global large solution by building a new quantity $\Gamma=w-\frac{curl div}{\Delta}\tau$ and finding a $H^1$ estimation for small initial data:
\begin{equation*}
  a\|u(t)\|_{L^{\infty}_t(H^1)}+\|\tau(t)\|_{L^{\infty}_t(H^1)}
+\|\tau\|_{L^2_t(H^2)}\leq C\|(u_0,\tau_0)\|_{H^1}
\end{equation*}
Recently, Deng, Luo and Yin \cite{deng2021global} obtain the global small solutions for \eqref{gs2222} with $\mathrm{Q}(\tau,\nabla u)=\tau\Omega(u)-\Omega(u)\tau+b(\mathrm{D}(u)\tau+\tau \mathrm{D}(u)),~b=0$, and give an estimation of the $H^1$ decay for the global solutions with any $b\in [-1,1]$ . Combining their results, and after some calculation, one can get the following theorem:
\begin{theorem}\cite{ER,deng2021global}\label{Th1}
Let $(u_0,\tau_0)\in H^s(\mathbb{R}^2),s>2$. Assume that
$$\|u_0\|_{H^1}\leq \epsilon_0,~\|\tau_0\|_{H^1}\leq a\epsilon_0~~and~~\|w_0\|_{{{B}^{0}_{\infty,1}}}
+\|\tau_0\|_{{{B}^{0}_{\infty,1}}}\leq \varepsilon_0,$$
then the systems \eqref{gs2222} exists a unique global solution $(u,\tau)\in C([0,+\infty),H^s)$. Moreover, if $(u_0,\tau_0)\in L^1$, then we have
$$\|(u,\tau)(t)\|_{H^1}\leq C\|(u_0,\tau_0)\|_{H^1}[a(1+t)]^{-\frac{1}{2}}.$$
\end{theorem}
In this paper, we consider the global existence for \eqref{gs2222} with $a=0$ in critical Besov space. The main difficulty is that the ratio is 1 ($a=0$). Since \cite{wjh8} tells us that the term $a\|\mathbb{D}u\|_{H^s}$ has disspative effect, which leads to the polynomial decay of $\|\mathbb{D}u\|_{H^s},~s\geq 0.$ This implies that when $a=0$, $\|\mathbb{D}u\|_{H^s}$ may grow! That causes the main difficulty to obtain the global existence of \eqref{gs2222}. However, we find that $\|\tau(t)\|_{L^{2}\cap L^{\infty}}$  is exponential decay for small initial data, which can offset the exponential growth effect due to the Gronwall inequality. Then, combing the critical estimation of the transport equation in Besov space and the bootstrap argument, we obtain the global existence of \eqref{gs2222}. Compared with the results in \cite{ER,deng2021global}, we needn't the small condition of $\|u_0,\tau_0\|_{\dot{H}^1}$. Finally, we obtain the large time behavior such that $\|\tau(t)\|_{B^{\frac{d}{p}}_{p,1}}$ and $\|u(t)\|_{B^{1+\frac{d}{p}}_{p,1}}$ are double exponential growth~(see \eqref{see1}-\eqref{see2} ). Here is our first result:
  \begin{theorem}\label{Th2}
 	 {\it
 Let $(u_0,\tau_0)\in (B^{1+\frac{2}{p}}_{p,1},B^{\frac{2}{p}}_{p,1}),~1\leq p\leq\infty$. If
\begin{equation}\label{th2}
  \|(u_0,\tau_0)\|_{L^2}+\|w_0\|_{{{B}^{0}_{\infty,1}}}
  +\|\tau_0\|_{{{B}^{0}_{\infty,1}}}\leq 4\varepsilon_0,~~w_0=\mathrm{curl}u_0,~ \varepsilon_0=\frac{1}{64(C^4+1)},
\end{equation}
then there exists a unique global solution $(u,\tau)\in C_T(B^{1+\frac{d}{p}}_{p,1})\times C_T(B^{\frac{d}{p}}_{p,1})\cup L^1_T(B^{2+\frac{d}{p}}_{p,1})$ for system \eqref{gs2222} with $a=0$. Moreover, we have
$$\|\tau\|_{L^2\cap L^{\infty}}\leq C e^{-\frac{1}{72}t}\|\tau_0\|_{L^2\cap L^{\infty}}.$$
}
\end{theorem}
\begin{remark}\label{r1}
In \cite{Bourgain1,Bourgain2}, Bourgain and Li employed a combination of Lagrangian and Eulerian techniques to obtain strong local ill-posed results of the Euler
equation in $B^{\frac{d}{p}+1}_{p,r}$ with $p\in[1,\infty),~r\in(1,\infty],~d=2,3$. Recently, Guo, Li and Yin \cite{Lijfa} proved the Euler
equation is well-posed in $B^{\frac{d}{p}+1}_{p,1}$ with $p\in[1,\infty]$, which means that $B^{\frac{d}{p}+1}_{p,1}$ is the critical Besov space for the well-posedness of the Euler equation. Thus, since \eqref{gs2222} is the Euler equation when $\tau_0=0$, we conclude that $C([0,T];{B}^{\frac{d}{p}+1}_{p,1}(\mathbb{R}^d))\times \Big(C([0,T];{B}^{\frac{d}{p}}_{p,1}(\mathbb{R}^d))\cap L^1([0,T];{B}^{\frac{d}{p}+2}_{p,1}(\mathbb{R}^d))\Big)$ is also the critical Besov space for \eqref{gs2222} with $d=2,3,~p\in[1,\infty)$.
\end{remark}
\begin{remark}
Since $B^{\frac{2}{p}}_{p,1}\hookrightarrow B^{0}_{\infty,1},~p<\infty$. By \eqref{th2} we claim that our result includes some large initial data. For example,
choose $\varphi$ be a smooth, radial and non-negative function in $\mathbb{R}^2$ such that
\begin{equation}
\phi=\begin{cases}
1,~~for~|\xi|\leq 1,\\
0,~~for~|\xi|\geq 2.
\end{cases}
\end{equation}
Let $(u_0,\tau_0):=\frac{1}{N}(\psi,\varphi)$, where $\psi,\varphi\in S(\mathbb{R}^2)$, $div\psi=0$ and $F(\varphi)=(\phi(\xi-2^Ne),\phi(\xi-2^Ne))$ with $e=(1,1),N\in\mathbb{N}^+$. Then, one can easily deduce that
 $$\Delta_j\varphi=\varphi~~when~j=N;
 ~~~\Delta_j\varphi=0~~when~j\neq N.$$
So for sufficient large $N$ and $p<\infty$, we have
$$\|\tau_0\|_{H^1}\approx\frac{2^{N}}{N},~~
\|\tau_0\|_{B^{\frac{2}{p}}_{p,1}}\approx\frac{2^{\frac{2}{p}N}}{N},
~~but~~\|\tau_0\|_{B^{0}_{\infty,1}}+\|\tau_0\|_{L^2}\leq \frac{C}{N}.$$
This implies the global existence for some large initial data. Combining Remark \ref{r1}, we obtain the global large solutions for \eqref{gs2222} with $a=0$ in critical Besov spaces, which implies some additions and improvements to \cite{ER,deng2021global}.
\end{remark}

\par
Next, by Theorem \ref{Th1} ($a>0$) and Theorem \ref{Th2} ($a=0$), we obtain the global existence for the system \eqref{gs2222}. Furthermore, one will ask that whether these two solutions are close to each other when $a\rightarrow 0$? For local time the answer is true. That is
$$\lim_{a\rightarrow 0}\|u^0-u^a\|_{L^{\infty}_T(B^{2}_{2,1})}
 +\|\tau^0-\tau^a\|_{L^{\infty}_T(B^{1}_{2,1})}=0,~~~\forall t\in [0,T].$$
However, for large time, the above equality is no longer valid. Indeed, $\|u\|_{L^2}~({B^2_{2,1}}\hookrightarrow {L^2})$ will have a jump when $T$ is sufficiently large. This implies the system \eqref{gs2222} is not steady for $a\rightarrow 0$ for large time, while for local time \eqref{gs2222} is steady. Therefore, whether the ratio is 1 ($a=0$) seems to have a fundamental effect on the nature of the global solutions for \eqref{gs2222}.
\par
In order to give the Theorem 1.3, we introduced a simple notation.
\par
Set $\mathbb{A}:=\{(u_0,\tau_0)\in (B^{2}_{2,1} (\mathbb{R}^2),B^{1}_{2,1} (\mathbb{R}^2))|~\eqref{gs2222}~ has~a~unique~solution~for~any~fixed~a\}.$ Here is our second result.
\begin{theorem}\label{Th3}
 Let $(u_0,\tau_0)\in (B^{2}_{2,1}(\mathbb{R}^2),B^{1}_{2,1} (\mathbb{R}^2))$
 as initial data of systems \eqref{gs2222}, $(u^a,\tau^a)$ is the corresponding solution for \eqref{gs2222} with every $a\geq 0$.\\
 (1) Then there exists a lifespan $T(u_0,\tau_0)$ independent of $a$ such that
\begin{align}\label{th3.1}
 \lim_{a\rightarrow 0}\|u^0-u^a\|_{L^{\infty}_T(B^{2}_{2,1})}
 +\|\tau^0-\tau^a\|_{L^{\infty}_T(B^{1}_{2,1})}=0.
\end{align}
(2) However, \eqref{th3.1} false for large time. Assume
 $$0< a\leq \varepsilon_0.$$
 Then there exists a time $T(a)$ and a sequence  $(u_0,\tau_0)\in\mathbb{A}$ as initial data of systems \eqref{gs2222} such that when $t\geq T(a)$,
 $$ \|(u^0-u^a)(t)\|_{L^2}\geq\frac{\varepsilon_0}{8},$$
 where $\varepsilon_0=\frac{1}{64(C^4+1)}  $, a fixed constant.
 \end{theorem}
 \begin{remark}
When $a=0$, by Theorem \ref{Th2}, we deduce that $\|u\|_{L^2}$ is bounded, while $\|\tau\|_{L^2}$ is exponential decay. However, when $a>0$, by Theorem \ref{Th1}, we obtain the polynomial decay of $\|u\|_{L^2}$ and $\|\tau\|_{L^2}$. This interesting fact implies that the ratio has two sides to the effect. It passes the decay of $\tau$ to $u$, while $\tau$ decays from exponential attenuation to polynomial attenuation. This shows that the energy of the velocity field and the stress tensor are converted to each other, but the process of transformation is transient for large time (see Theorem \ref{Th3}).
 \end{remark}

In the case $\alpha=0,$ we were surprised to find that the $\tau$ have a exponential decay by {iteration and reasonable assumption of bootstrap,} using the integral equation of the heat
 equation and bootstrap assumption, we can estimate the $L^\infty$ norm of $\tau$. Combining the prior estimate of transport equation
 in Besov space and subtle analytical skills, we complete the proof of the Theorem \ref{Th2}. Moreover, on the basis of Theorem \ref{Th2}, we
 give the proof of the Theorem \ref{Th3} by boothstrap argument.
\par\noindent
{\bf Notation}
Throughout the paper, we denote the norms of usual Lebesgue space $L^p(\R^3)$ by $
\|u\|^p_{L^p}=\int_{\Omega}|u|^p dx,~\hbox{for}~1\leq p<\infty$.
$C_i$ and $C$ denote different
positive constants in different places.
\par
The bootstrap argument will be employed in our proof. A rigorous statement of the abstract bootstrap principle can be found in T. Tao's book (see \cite{TT2}).
\par
The paper is organized as follows. In section 2, we will give the tools(Littlewood-Paley decomposition and
paradifferential calculus) and some useful prior estimate in Besov spaces. In section 3,  we give the proof of The theorem \ref{Th2}. In last section, we complete the
proof of unstability results.

\section{Preliminaries}
\par
In this section, we will recall some properties about the Littlewood-Paley decomposition and Besov spaces. {For more details, we refer the readers to the \cite{CH}.}
\begin{proposition}\label{p2}\cite{CH}
Let $\mathcal{C}$ be the annulus $\{\xi\in\mathbb{R}^d:\frac 3 4\leq|\xi|\leq\frac 8 3\}$. There exist radial functions $\chi$ and $\varphi$, valued in the interval $[0,1]$, belonging respectively to $\mathcal{D}(B(0,\frac 4 3))$ and $\mathcal{D}(\mathcal{C})$, and such that
$$ \forall\xi\in\mathbb{R}^d,\ \chi(\xi)+\sum_{j\geq 0}\varphi(2^{-j}\xi)=1, $$
$$ \forall\xi\in\mathbb{R}^d\backslash\{0\},\ \sum_{j\in\mathbb{Z}}\varphi(2^{-j}\xi)=1, $$
$$ |j-j'|\geq 2\Rightarrow\mathrm{Supp}\ \varphi(2^{-j}\cdot)\cap \mathrm{Supp}\ \varphi(2^{-j'}\cdot)=\emptyset, $$
$$ j\geq 1\Rightarrow\mathrm{Supp}\ \chi(\cdot)\cap \mathrm{Supp}\ \varphi(2^{-j}\cdot)=\emptyset. $$
The set $\widetilde{\mathcal{C}}=B(0,\frac 2 3)+\mathcal{C}$ is an annulus, and we have
$$ |j-j'|\geq 5\Rightarrow 2^{j}\mathcal{C}\cap 2^{j'}\widetilde{\mathcal{C}}=\emptyset. $$
Further, {one has}
$$ \forall\xi\in\mathbb{R}^d,\ \frac 1 2\leq\chi^2(\xi)+\sum_{j\geq 0}\varphi^2(2^{-j}\xi)\leq 1, $$
$$ \forall\xi\in\mathbb{R}^d\backslash\{0\},\ \frac 1 2\leq\sum_{j\in\mathbb{Z}}\varphi^2(2^{-j}\xi)\leq 1. $$
\end{proposition}

\begin{definition}\cite{CH}
Let $u$ be a tempered distribution in $\mathcal{S}'(\mathbb{R}^d)$ and $\mathcal{F}$ be the Fourier transform and $\mathcal{F}^{-1}$ be its inverse. For all $j\in\mathbb{Z}$, define
$$
\Delta_j u=0\,\ \text{if}\,\ j\leq -2,\quad
\Delta_{-1} u=\mathcal{F}^{-1}(\chi\mathcal{F}u),\quad
\Delta_j u=\mathcal{F}^{-1}(\varphi(2^{-j}\cdot)\mathcal{F}u)\,\ \text{if}\,\ j\geq 0,\quad
S_j u=\sum_{j'<j}\Delta_{j'}u.
$$
Then the { nonhomogeneous} Littlewood-Paley decomposition is given as follows:
\begin{align*}
u=\sum_{j\in\mathbb{Z}}\Delta_j u \quad \text{in}\ \mathcal{S}'(\mathbb{R}^d).	
\end{align*}
Let $s\in\mathbb{R},\ 1\leq p,r\leq\infty.$ The nonhomogeneous Besov space $B^s_{p,r}(\mathbb{R}^d)$ is defined by
$$ B^s_{p,r}=B^s_{p,r}(\mathbb{R}^d)=\{u\in S'(\mathbb{R}^d):\|u\|_{B^s_{p,r}(\mathbb{R}^d)}=\Big\|(2^{js}\|\Delta_j u\|_{L^p(\mathbb{S}^d)})_j \Big\|_{l^r(\mathbb{Z})}<\infty\}. $$
\end{definition}
\begin{definition}\cite{CH}
The homogeneous dyadic blocks $\dot{\Delta}_j$ are defined on the tempered distributions by
\begin{equation*}
  \dot{\Delta}_ju=\varphi(2^{-j}D)u:=\mathcal{F}^{-1}(\varphi(2^{-j\cdot})\hat{u}).
\end{equation*}
\begin{equation*}
  \dot{S}_ju=\sum_{j'\leq j-1}\dot{\Delta}_{j'}u.
\end{equation*}
 \begin{definition}\label{de1}
 	  We denote by $S'_h$ the space of tempered distributions $u$ such that
 	 \begin{equation*}
 	 \lim_{j\rightarrow-\infty}\dot{S}_ju=0 ~~\text{in}~~S'.
 	 \end{equation*}

 \end{definition}
 The homogeneous Littlewood-Paley decomposition is defined as
 \begin{equation*}
 u=\sum_{j\in\mathbb{Z}}\dot{\Delta}_ju,~~~~\text{for}~~u\in S'_h.
 \end{equation*}
\end{definition}
\begin{definition}
  For $s\in\R, 1\leq p\leq\infty,$ the homogeneous Besov space $\dot{B}^s_{p,r}$ is defined as
\begin{equation*}
  \dot{B}^s_{p,r}:=\{u\in S'_h, \|u\|_{\dot{B}^s_{p,r}}<\infty\},
\end{equation*}
where the homogeneous Besov norm is given by
\begin{equation*}
  \|u\|_{\dot{B}^s_{p,r}}:=\|\{2^{js}\|\dot{\Delta}_ju\|_{L^p}\}_j\|_{l^r}.
\end{equation*}
\end{definition}

In this paper, we use the "time-space" Besov spaces or Chemin-Lerner space first introduced by Chemin and Lerner in \cite{CL}.
\par\noindent
{\bf Definition 2.3.}\,\,
{\it Let $s\in\R$ and $0<T\leq+\infty.$ We define
\begin{equation*}
  \|u\|_{\tilde{L}^q_T({B}^s_{p,1})}:=\sum_{j\in\mathbb{Z}}2^{js}\bigg(\int^T_0\|{\Delta}_ju(t)\|^q_{L^p}dt\bigg)^\frac{1}{q},
\end{equation*}
for $p,q\in[1,\infty)$ and with the standard modification for $p,q=\infty.$
}
\par\noindent
By the Minkowski's inequality, it is easy to verify that
\begin{equation*}
  \|u\|_{\tilde{L}^\lambda_T({B}^s_{p,r}))}\leq \|u\|_{L^\lambda_T({B}^s_{p,r}))}~~~\text{if}~~\lambda\leq r,
\end{equation*}
and
\begin{equation*}
  \|u\|_{\tilde{L}^\lambda_T({B}^s_{p,r}))}\geq \|u\|_{L^\lambda_T({B}^s_{p,r}))}~~~\text{if}~~\lambda\geq r.
\end{equation*}
\par\noindent
The following Bernstein's lemma will be repeatedly used in this paper.
 \begin{lemma}\label{le1}
 	 Let $\mathcal{B}$ is a ball and $\mathcal{C}$ is a ring of $\R^d.$ There exists constant $C$ such that for any positive $\lambda,$
 	 any non-negative integer $k,$ any smooth homogeneous function $\sigma$ of degree $m,$
 	 any couple $(p,q)\in[1,\infty]^2$ with $q\geq p\geq1,$ and any function $u\in L^p,$ there holds
 	 \begin{equation*}
 	 \mathrm{supp}\hat{u}\subset\lambda\mathcal{B}\Rightarrow\sup_{|\alpha=k|}\|\partial^\alpha u\|_{L^q}\leq C^{k+1}\lambda^{k+d(\frac{1}{p}-\frac{1}{q})}\|u\|_{L^p},
 	 \end{equation*}
 	 \begin{equation*}
 	 \mathrm{supp}\hat{u}\subset\lambda\mathcal{C}\Rightarrow C^{-k-1}\lambda^{k}\|u\|_{L^p}\leq\sup_{|\alpha=k|}\|\partial^\alpha u\|_{L^p}\leq C^{k+1}\lambda^{k}\|u\|_{L^p},
 	 \end{equation*}
 	 \begin{equation*}
 	 \mathrm{supp}\hat{u}\subset\lambda\mathcal{C}\Rightarrow \sup_{|\alpha=k|}\|\sigma(D)u\|_{L^p}\leq C_{\sigma,m}\lambda^{m+d(\frac{1}{p}-\frac{1}{q})}\|u\|_{L^p}.
 	 \end{equation*}
\end{lemma}
Next, we will give the paraproducts and product estimates in Besov spaces. Recalling the paraproduct decomposition
\begin{equation*}
  uv={T}_uv+{T}_vu+{R}(u,v),
\end{equation*}
where
\begin{equation*}
  {T}_uv:=\sum_q{S}_{q-1}u{\Delta}_v,~~~{R}(u,v):=\sum_q{\Delta}_qu{\tilde{\Delta}}_qv,~~\text{and}~~{\tilde{\Delta}}_q
  ={\Delta}_{q-1}+{\Delta}_{q}+{\Delta}_{q+1}.
\end{equation*}
\par\noindent
The paraproduct ${T}$ and the remainder ${R}$ operators satisfy the following continuous properties.
\begin{proposition}\label{p2}
  For all $s\in\R, \sigma>0,$ and $1\leq p, p_1,p_2,r,r_1,r_2\leq\infty,$ the paraproduct ${T}$ is a bilinear,
continuous operator from $L^\infty\times{B}^s_{p,r}$ to ${B}^s_{p,r}$ and from
${B}^{-\sigma}_{p_1,r_1}\times\dot{B}^{s}_{p_2,r_2}$ to ${B}^{s-\sigma}_{p,r}$ with $\frac{1}{r}=\min\{1,\frac{1}{r_1}+\frac{1}{r_2}\},
\frac{1}{p}=\frac{1}{p_1}+\frac{1}{p_2}.$
The remainder ${R}$ is bilinear continuous from ${B}^{s_1}_{p_1,r_1}\times {B}^{s_2}_{p_2,r_2}$
 to ${B}^{s_1+s_2}_{p,r}$ with $s_1+s_2>0, \frac{1}{p}=\frac{1}{p_1}+\frac{1}{p_2}\leq1,$
 and $\frac{1}{r}=\frac{1}{r_1}+\frac{1}{r_2}\leq 1.$ In particular, if $r=\infty,$
the continuous property for the remainder ${R}$ also holds for the case $s_1+s_2=0, r=\infty, \frac{1}{r_1}+\frac{1}{r_2}=1.$
\end{proposition}
Combining the above proposition with Lemma \ref{le1} yields the following product estimates:

\begin{coro}\label{co1}
	 Let $a$ and $b$ be in $L^\infty\cap {B}^s_{p,r}$ for some $s>0$ and $(p,r)\in[1,\infty]^2.$
	Then there exists a constant $C$ depending only on $d, p$ and such that
	\begin{equation*}
	\|ab\|_{{B}^s_{p,r}}\leq C(\|a\|_{L^\infty}\|b\|_{{B}^s_{p,r}}+\|b\|_{L^\infty}\|a\|_{{B}^s_{p,r}}).
	\end{equation*}
\end{coro}
{We note that the above Proposition and Corollary is valid in the homogeneous framework (i.e. with $\dot{\Delta}_j$ instead of $\Delta_j$ and with
homogeneous Besov norms instead of nonhomogeneous ones), the readers can refer the \cite{CH}.}
Finally, we intruduce some useful results about the following heat conductive equation and the transport equation
\begin{equation}\label{s1cuchong}
\left\{\begin{array}{l}
    u_t-\Delta u+\beta u=G,\ x\in\mathbb{R}^d,\ \beta\geq0,~t>0, \\
    u(0,x)=u_0(x),\ x\in\mathbb{R}^d,
\end{array}\right.
\end{equation}
\begin{equation}\label{s100}
\left\{\begin{array}{l}
    f_t+v\cdot\nabla f=g,\ x\in\mathbb{R}^d,\ t>0, \\
    f(0,x)=f_0(x),\ x\in\mathbb{R}^d,
\end{array}\right.
\end{equation}
which are crucial to the proof of our main theorem later.
\begin{lemma}\label{priori estimate111}\cite{MR4199961}
Let $1\leq p\leq q\leq\infty$ and $k\geq 0,$ it holds that
\begin{equation*}
  \|\nabla^ke^{t\Delta}f\|_{L^{q}}\leq Ct^{-\frac{k}{2}-\frac{1}{p}+\frac{1}{q}}\|f\|_{L^{p}}.
\end{equation*}
\end{lemma}

\begin{lemma}\label{heat}\cite{CH}
Let $s\in\mathbb{R},~\beta\geq 0, 1\leq q,q_1,p,r\leq\infty$ with $q_1\leq q$. Assume $u_0$ in ${B}^s_{p,r}$, and $G$ in $\widetilde{L}^{q_1}_T(^s_{p,r})$. Then \eqref{s1cuchong} with $\beta=0$ has a unique solution $u$ in $\widetilde{L}^{q}_T({B}^{s+\frac{2}{q}}_{p,r})$ and satisfies
\begin{equation}\label{heatg1}
\aligned
\|u\|_{\widetilde{L}^{q}_T({B}^{s+\frac{2}{q}}_{p,r})}\leq C_1\Big(\|u_0\|_{{B}^s_{p,r}}+(1+T^{1+\frac{1}{q}-\frac{1}{q_1}})\|G\|_{\widetilde{L}^{q_1}_T({B}^{s+\frac{2}{q_1}-2}_{p,r})}\Big).
\endaligned
\end{equation}
Moreover, if $\beta>0$, without loss of generality we set $\beta=1$, one have
\begin{equation}\label{heatg2}
\aligned
\|u\|_{\widetilde{L}^{q}_T({B}^{s+\frac{2}{q}}_{p,r})}\leq C_1\Big(\|u_0\|_{{B}^s_{p,r}}+\|G\|_{\widetilde{L}^{q_1}_T({B}^{s+\frac{2}{q_1}-2}_{p,r})}\Big). \endaligned
\end{equation}
\end{lemma}
\begin{proof}
\eqref{heatg1} can be founded in \cite{CH}, we should only prove \eqref{heatg2}. Indeed, since
$$\Delta_ju=e^{-t}e^{t\Delta}\Delta_ju_0
+\int_0^te^{-(t-s)}e^{(t-s)\Delta}\Delta_jGds,$$
when $j\geq 0$, by $\|e^{t\Delta}\Delta_ju\|_{L^{p}}\leq Ce^{-2^{2j}t}\|\Delta_ju\|_{L^{p}}$ one can easily get
$$\|2^{j(s+\frac{2}{q})}\|\Delta_ju\|_{L^{q}_TL^{p}}\|_{1_{j\geq0}l^r}\leq C_1\Big(\|u_0\|_{{B}^s_{p,r}}+\|G\|_{\widetilde{L}^{q_1}_T({B}^{s+\frac{2}{q_1}-2}_{p,r})}\Big).$$
Indeed, we just estimate the difficult term
\begin{equation*}
  \sum_{j\geq0}2^{j(s+\frac{2}{q})}(\int^T_0|\|\int^t_0e^{-(t-s)}e^{(t-s)\Delta}\Delta_jGds\|_{L^p}|^qdt)^{\frac{1}{q}},
\end{equation*}
Combining {the Minkowski's inequality and the Young inequality, one has}
\begin{equation*}
\aligned
  \sum_{j\geq0}&2^{j(s+\frac{2}{q})}(\int^T_0|\int^t_0e^{-(t-s)}\|e^{(t-s)\Delta}\Delta_jG\|_{L^p}ds|^qdt)^{\frac{1}{q}}\\
  &\leq\sum_{j\geq0}2^{j(s+\frac{2}{q})}(\int^T_0|e^{-(2^{2j}+1)\cdot}*\|\Delta_jG(\cdot)\|_{L^p}|^q)^{\frac{1}{q}}\\
  &\leq\sum_{j\geq0}2^{j(s+\frac{2}{q})}(\int^T_0|e^{-(2^{2j}+1)t}|^{\frac{1}{1+\frac{1}{q}-\frac{1}{q_1}}}dt)^{1+\frac{1}{q}-\frac{1}{q_1}}\|\Delta_jG\|_{L^{q_1}_T(L^p)}\\
  &\leq C\sum_{j\geq0}2^{j(s+\frac{2}{q}-2(1+\frac{1}{q}-\frac{1}{q_1}))}\|\Delta_jG\|_{L^{q_1}_T(L^p)}\\
  &=C\sum_{j\geq0}2^{j(s-2+\frac{2}{q_1})}\|\Delta_jG\|_{L^{q_1}_T(L^p)}.
\endaligned
\end{equation*}
When $j=-1$, by $\|e^{t\Delta}\Delta_{-1}u\|_{L^{p}}\leq C\|\Delta_{-1}u\|_{L^{p}}$ we have
$$\|\Delta_{-1}u\|_{L^{q}_TL^{p}}\leq C_1\Big(\|\Delta_{-1}u_0\|_{L^p}+\|\Delta_{-1}G\|_{{L}^{q_1}_T(L^p)}\Big).$$
Combining the above two inequality, we obtain \eqref{heatg2}.
\end{proof}

\begin{lemma}\label{priori estimate}\cite{CH}
Let $s\in [\max\{-\frac{d}{p},-\frac{d}{p'}\},\frac{d}{p}+1](s=1+\frac{d}{p},r=1; s=\max\{-\frac{d}{p},-\frac{d}{p'}\},r=\infty).$
There exists a constant $C$ such that for all solutions $f\in L^{\infty}([0,T];{B}^s_{p,r})$ of \eqref{s100} with initial data $f_0$ in ${B}^s_{p,r}$, and $g$ in $L^1([0,T];{B}^s_{p,r})$, we have, for a.e. $t\in[0,T]$,
\begin{equation}\label{g1}
\aligned
\|f(t)\|_{{B}^s_{p,r}}\leq& C\Big(\|f_0\|_{{B}^s_{p,r}}+\int_0^t V'(t')\|f(t')\|_{{B}^s_{p,r}}+\|g(t')\|_{{B}^s_{p,r}}dt' \Big)\\
\leq&
e^{C_2 V(t)}\Big(\|f_0\|_{{B}^s_{p,r}}+\int_0^t e^{-C_2 V(t')}\|g(t')\|_{{B}^s_{p,r}}dt'\Big),
\endaligned
\end{equation}
where $V(t)=\int_{0}^{t}\|\nabla v\|_{{B}^{\frac{d}{p}}_{p,r}\cap L^{\infty}}ds($if $s=1+\frac{1}{p},r=1$, $V'(t)=\int_{0}^{t}\|\nabla v\|_{{B}^{\frac{d}{p}}_{p,1}}ds).$
\end{lemma}

\begin{remark}\label{priori estimate1}\cite{CH}
If ${\rm{div}}v=0$, we can get the same result with a better indicator: $\max\{-\frac{d}{p},-\frac{d}{p'}\}-1<s<\frac{d}{p}+1($or $s=\max\{-\frac{d}{p},-\frac{d}{p'}\}-1,r=\infty).$
\end{remark}

\begin{lemma}\label{priori estimate0}\cite{CH}
Let ${\rm{div}} v=0$.
There exists a constant $C$ such that for solutions $f\in L^{\infty}([0,T];{B}^0_{p,r})$ of \eqref{s100} with initial data $f_0$ in ${B}^0_{p,r}$, and $g$ in $L^1([0,T];{B}^s_{p,r})$, we have, for all $1\leq p,r\leq \infty$ and $t\in[0,T]$,
\begin{align}
\|f(t)\|_{{B}^s_{p,r}}\leq& C(1+\int_0^t V'(t')dt')(\|f_0\|_{{B}^s_{p,r}}+\int_0^t\|g(t')\|_{{B}^s_{p,r}}dt'),
\end{align}
where $V'(t)=\int_{0}^{t}\|\nabla v(t)\|_{L^{\infty}}ds.$
\end{lemma}

\par\noindent
\section{Global existence}

\par\noindent
\textbf{The proof of Theorem \ref{Th2}:}
\begin{proof}
The proof of the local well-posedness of the system \ref{gs2222} is standard, we should only need to prove the global existence.
We now use the bootstrap argument to prove this Theorem.
Let $T^*$ be the maximal existence time of the solution. Assume that for any $t\leq T< T^*$,
\begin{align}\label{global00}
\|\tau\|_{L^{\infty}_T(B^{0}_{\infty,1})\cap L^{1}_T(B^{2}_{\infty,1})}\leq \delta:=\frac{1}{C^2+1},~~\|u\|^2_{L^{\infty}_T(L^2)}\leq \delta^2.
\end{align}

where $C>10$ is a fixed constant.
Let the sufficient small initial data such that
\begin{align}\label{globaljie1}
\|u_0\|_{B^1_{\infty,1}(\R^2)}
+\|b_0\|_{B^0_{\infty,1}(\R^2)}\leq\epsilon_0:=\frac{1}{64(C^4+1)},~~
\|u_0\|_{L^2}+\|\tau_0\|_{L^2}\leq\epsilon_0.
\end{align}
The proof can be divided into 3 parts:

(1) First, we give the decay estimation of $\|\tau(t)\|_{L^2}$.
\par\noindent
Recall the system:
\begin{equation}\label{beiyong}
\begin{cases}
\partial_tu+(u\cdot\nabla)u=\mathrm{div}(\tau),\\
\partial_t\tau+(u\cdot\nabla)\tau-\Delta\tau+\tau+\mathrm{Q}(\nabla u,\tau)=0,\\
\mathrm{div} u=0,\\
u(x,0)=u_0(x),~~~\tau(0,x)=\tau_0(x).
\end{cases}
\end{equation}
Taking the curl of the first equation in \eqref{beiyong}, we get:
\begin{equation}\label{gs3}
\begin{cases}
\partial_tw+(u\cdot\nabla)w=\mathrm{curl}\mathrm{div}(\tau),\\
w_0=\mathrm{curl}u_0.
\end{cases}
\end{equation}
It is easy to verify that
\begin{equation}\label{gs11}
  \|w\|_{L^\infty}\leq \|w_0\|_{L^\infty}+\int^t_0\|\tau\|_{B^{2}_{\infty,1}}ds\leq C(\epsilon_0+\delta).
\end{equation}
From the second equation of \eqref{beiyong}, we get
\begin{equation*}
  \partial_te^t\tau+(u\cdot\nabla)e^t\tau-\Delta e^t\tau=-Q(\nabla u,e^t\tau).
\end{equation*}
Taking the $L^2$ inner product with $e^t\tau,$
one has
\begin{equation*}
  \frac{1}{2}\frac{d}{dt}\|e^t\tau\|_{L^2}^2+(\nabla e^t\tau,\nabla e^t\tau)=-(Q(\nabla u,e^t\tau),e^t\tau),
\end{equation*}
and using the integration by parts, H$\ddot{o}$lder's inequalities and Young's inequalities,
\begin{equation}\label{x1}
\aligned
  \|e^t\tau\|^2_{L^2}+2\int^t_0\|e^s\tau\|^2_{\dot{H}_1}ds&\leq \|\tau_0\|^2_{L^2}+\int^t_0\|u\|_{L^\infty}(\|\nabla\tau\|_{L^2}\|\tau\|_{L^2})e^{2s}ds\\
  &\leq \|\tau_0\|^2_{L^2}+\int^t_0(\|w\|_{L^\infty}
  +\|u\|_{L^2})\|e^s\tau\|^2_{\dot{H}_1}\\&~~+C
 (\|w\|_{L^\infty}+\|u\|_{L^2})\|e^s\tau\|^2_{L^2}ds\\
 &\leq \epsilon_0^2+C\int^t_0(\epsilon_0+\delta)\|e^s\tau\|^2_{\dot{H}_1}\\&~~+C
 (\epsilon_0+\delta)\|e^s\tau\|^2_{L^2}ds
\endaligned
\end{equation}
where the last inequality holds by \eqref{global00} and \eqref{gs11}.
Combining the Gronwall's inequality with \eqref{globaljie1}, one has
\begin{equation}\label{gs13}
  \|e^t\tau\|^2_{L^2}+\int^t_0\|e^s\tau\|^2_{\dot{H}_1}ds\leq e^{\frac{1}{2}t}\|\tau_0\|^2_{L^2}.
\end{equation}
Taking the $L^2$ inner product again with $e^{\frac{1}{2}t}\tau,$ we have
\begin{equation*}
  \frac{1}{2}\frac{d}{dt}\|e^{\frac{1}{2}t}\tau\|_{L^2}^2+(e^{\frac{1}{2}t}\tau,e^{\frac{1}{2}t}\tau)+(\nabla e^{\frac{1}{2}t}\tau,\nabla e^{\frac{1}{2}t}\tau)=-(Q(\nabla u,e^{\frac{1}{2}t}\tau),e^{\frac{1}{2}t}\tau),
\end{equation*}
 Similar to the inequality \eqref{x1}, it is easy to see that
\begin{equation}\label{x2}
\aligned
  \|e^{\frac{1}{2}t}\tau\|^2_{L^2}+\int^t_0\|e^{\frac{1}{2}s}\tau\|^2_{\dot{H}_1}ds&\leq \|\tau_0\|^2_{L^2}+\int^t_0C
 (\|w\|_{L^\infty}+\|u\|_{L^2})\|\tau\|^2_{L^2}e^sds
\endaligned
\end{equation}
Noting the inequality \eqref{gs13}, it implies that,
\begin{equation}\label{x4}
  \|\tau\|^2_{L^2}\leq e^{-\frac{3}{2}t}\|\tau_0\|^2_{L^2}.
\end{equation}
Using the inequality \eqref{x2} and \eqref{x4}, one has
\begin{equation}\label{x5}
  \|e^{\frac{1}{2}t}\tau\|^2_{L^2}+\int^t_0\|e^{\frac{1}{2}s}\tau\|^2_{\dot{H}_1}ds\leq C\|\tau_0\|^2_{L^2}
\end{equation}

(2) Then, we give the estimation of $\|u(t)\|_{L^{\infty}_T(L^2)}\leq \frac{1}{2}\delta.$
\par\noindent
Taking the $L^2$ inner product with $u$ for the first equation of \eqref{gs2222},
we get
\begin{equation*}
\aligned
   \|u\|^2_{L^2}&\leq\|u_0\|^2_{L^2}+\int^t_0\|div\tau\|_{L^2}\|u\|_{L^2}ds\\
   &\leq\|u_0\|^2_{L^2}+\int^t_0e^{\frac{1}{2}s}\|div\tau\|_{L^2}e^{-\frac{1}{2}s}\|u\|_{L^2}ds\\
   &\leq\|u_0\|^2_{L^2}+(\int^t_0e^{s}\|div\tau\|^2_{L^2})^{\frac{1}{2}}(\int^t_0e^{-s}\|u\|^2_{L^2})^{\frac{1}{2}}\\
   &\leq\epsilon_0^2+C\epsilon_0\delta,
\endaligned
\end{equation*}
where we have used the fact the inequality $\|u\|^2_{L^{\infty}_T(L^2)}\leq \delta^2$ and \eqref{x5}.
We can set $\epsilon_0$  is small enough in \eqref{globaljie1} such that
\begin{equation}\label{L2}
  \|u\|_{L^2}\leq\frac{1}{2}\delta.
\end{equation}

(3) Next, we estimate  $\|w(t)\|_{B^{0}_{\infty,1}}$ and $\|\tau\|_{L^\infty}.$
\par\noindent
Applying the Lemma \ref{priori estimate0} to the first equation of \eqref{gs3},
one has
\begin{equation}\label{x6}
\aligned
  \|w\|_{L^\infty_T({B^{0}_{\infty,1}})}&\leq(\|w_0\|_{B^{0}_{\infty,1}}+C\|\tau\|_{L^1_T(B^{2}_{\infty,1})})(1+\int_{0}^{T}\|\nabla u\|_{L^{\infty}}ds)\\
&\leq C(\|w_0\|_{B^{0}_{\infty,1}}+\|\tau\|_{L^1(B^{2}_{\infty,1})})(1+\int_{0}^{T}\|w\|_{L^\infty_T({B^{0}_{\infty,1}})}+\|u\|_{L^2}ds)\\
&\leq C(\epsilon+\delta)+\delta T +\int_{0}^{T}C(\epsilon+\delta)\|w\|_{L^\infty_T({B^{0}_{\infty,1}})}ds\\
&\leq C(c_{small}+\delta T)e^{c_{small}T}
\endaligned
\end{equation}
where the last inequality holds by the Gronwall's inequality, and $c_{small}:=C(\epsilon+\delta)$.

Then, we estimate the $\|u\|_{L^\infty}.$ we can rewrite the second equation \eqref{beiyong} into the following equality:
\begin{equation*}
  \partial_t(e^t\tau)+u\cdot\nabla(e^t\tau)-\Delta(e^t\tau)=-Q(\nabla u,e^t\tau).
\end{equation*}
We use the following integral equation
\begin{equation*}
  \|e^t\tau\|_{L^\infty}=e^{t\Delta}\tau_0-\int^t_0e^{(t-s)\Delta}(u\cdot\nabla(e^s\tau))ds+\int^t_0e^{(t-s)\Delta}Q(\nabla u,e^s\tau)ds.
\end{equation*}
By Lemma \ref{priori estimate111} and H$\ddot{o}$lder's inequalities, let $b$ be a fixed small constant such that
$8c_{small}<<\frac{1}{36}\leq b<\frac{1}{6}$, we have
\begin{equation*}
  \aligned
\|\tau\|_{L^\infty}&\leq e^{-t}\|\tau_0\|_{L^\infty}+\int^t_0e^{-(t-s)}\|u\nabla \tau\|_{L^{\infty}}ds+\int_{0}^{t}e^{-(t-s)}(t-s)^{-\frac{1}{2}}\|\tau\nabla u\|_{L^{2}}ds\\
  &\leq e^{-t}\|\tau_0\|_{L^\infty}+\int^t_0e^{-(t-s)}\|\tau\|^{\frac{1}{3}}_{L^{2}}
  e^{\frac{1}{6}s}\|\tau\|^{\frac{2}{3}}_{B^{2}_{\infty,1}}e^{-\frac{1}{6}s}\delta ds+\int_{0}^{t}e^{-(t-s)}(t-s)^{-\frac{1}{2}}\|\tau\nabla u\|_{L^{2}}ds\\
  &\leq e^{-t}\|\tau_0\|_{L^\infty}+e^{-tb}\int^t_0e^{sb}\|\tau\|^{\frac{1}{3}}_{L^{2}}
  e^{\frac{1}{6}s}\|\tau\|^{\frac{2}{3}}_{B^{2}_{\infty,1}}e^{-\frac{1}{6}s}\delta ds+\int_{0}^{t}e^{-b(t-s)}(t-s)^{-\frac{1}{2}}\|\tau\|_{L^{2}}\|\nabla u\|_{L^{\infty}}ds\\
  &\leq e^{-t}\epsilon_0+e^{-tb}C\|\tau_0\|_{L^2}^{\frac{1}{3}}\|\tau\|^{\frac{2}{3}}_{L^1_{B^{2}_{\infty,1}}}(\int^t_0|e^{sb}e^{-\frac{1}{6}s}|^3ds)^{\frac{1}{3}}\delta
  \\&~~~~~~~+e^{-tb}\int^t_0e^{sb}(t-s)^{-\frac{1}{2}}(c_{small}+\delta s)e^{c_{small}s}\epsilon_0 e^{-\frac{3}{4}s}ds\\
  &\leq e^{-t}\epsilon_0+C(t+1)e^{-tb}\delta\\
  &\leq C(1+t)e^{-tb}\delta\leq Ce^{-\frac{b}{2}t}\delta,
  \endaligned
\end{equation*}
where the second inequality holds by the interpolation inequality:
\begin{equation*}
\aligned
\|u\nabla\tau\|_{L^\infty}&\leq\|\nabla \tau\|_{L^\infty}\|u\|_{L^\infty}\\
&\leq\|\tau\|^{\frac{1}{3}}_{B^{-1}_{\infty,\infty}}\|\tau\|^{\frac{2}{3}}_{B^{2}_{\infty,\infty}}(\|u\|_{L^2}+\|w\|_{L^\infty})\\
&\leq C\|\tau\|^{\frac{1}{3}}_{B^{-1}_{\infty,1}}\|\tau\|^{\frac{2}{3}}_{B^{2}_{\infty,1}}\delta\\
&\leq C\|\tau\|^{\frac{1}{3}}_{L^2}\|\tau\|^{\frac{2}{3}}_{B^{2}_{\infty,1}}\delta.
\endaligned
\end{equation*}

(4) Finally, to complete the bootstrap argument , we estimate  $\|\tau\|_{L^{\infty}_T(B^{0}_{\infty,1})\cap L^{1}_T(B^{2}_{\infty,1})}$ .
\par\noindent
Using the inequality \eqref{heatg2} of the  Lemma \eqref{heat} to the second equation of \eqref{gs3}, one has
\begin{equation*}
\aligned
  \|\tau\|_{L^{\infty}_T(B^{0}_{\infty,1})\cap L^{1}_T(B^{2}_{\infty,1})}&\leq\|\tau_0\|_{B^{0}_{\infty,1}}+\int^t_0\|u\nabla\tau\|_{B^{0}_{\infty,1}}+\|Q(\nabla u,
  \tau)\|_{B^{0}_{\infty,1}}ds\\
  &\leq \epsilon_0+\int^t_0\|u\|_{L^2}\|\tau\|_{B^{2}_{\infty,1}}+\|\nabla u\|_{B^{0}_{\infty,1}}\|\tau\|_{L^\infty}ds\\
  &\leq \epsilon_0+\frac{1}{2}\delta\int^t_0\|\tau\|_{B^{2}_{\infty,1}}ds+\int^t_0[(c_{small}+\delta s)e^{c_{small}s}+\delta]
  Ce^{-\frac{b}{2}s}\delta ds\\
  &\leq \epsilon_0+\frac{1}{2}\delta^2+C(c_{small}+\delta)\delta\int^t_0(1+s)e^{c_{small}s}e^{-\frac{b}{2}s}ds\\
  &\leq \frac{1}{2}\delta
\endaligned
\end{equation*}
where we have used the fact that
\begin{equation*}
\aligned
  \|u\cdot\nabla \tau\|_{B^{0}_{\infty,1}}+\|Q(\nabla u,\tau)\|_{B^{0}_{\infty,1}}\leq\|u\|_{L^2}\|\tau\|_{B^{2}_{\infty,1}}+\|\nabla u\|_{B^{0}_{\infty,1}}\|\tau\|_{L^\infty},
\endaligned
\end{equation*}
and the last inequality holds by \eqref{globaljie1} and $b\geq \frac{1}{36}>>8c_{small}$.\\

Combining \eqref{L2}, for any $t<T^*$ we have
\begin{align}\label{global000}
\|u\|^2_{L^{\infty}_T(L^2)}\leq \delta^2~~and~~\|\tau\|_{L^{\infty}_T(B^{0}_{\infty,1})\cap L^{1}_T(B^{2}_{\infty,1})}\leq \delta.
\end{align}

Therefore, one can obtain the global existence of $(u,\tau)$ in $ C([0,\infty); {B}^{1+\frac{2}{p}}_{p,1})\times \Big(C([0,\infty);{B}^{\frac{2}{p}}_{p,1})\cap L^1\big([0,\infty);{B}^{\frac{2}{p}+2}_{p,1}\big)\Big)$ easily, since \eqref{global000} can be the blow-up criteria for \eqref{gs2222}. Indeed, applying Lemma \ref{heat}--\ref{priori estimate} to \eqref{gs2222}, for any $t<T^*$ we have
\begin{align}\label{see1}
\|u\|_{L^{\infty}_{t}({B^{1+\frac{2}{p}}_{p,1}})}\leq  \|u_0\|_{{B^{1+\frac{2}{p}}_{p,1}}}
+C\int_{0}^{t}\|u\|_{B^1_{\infty,1}}\|u\|_{B^{1+\frac{2}{p}}_{p,1}}+\|\tau\|_{B^{1+\frac{2}{p}}_{p,1}}ds
\leq Ce^{e^{Ct}},
\end{align}
and
\begin{align}\label{see2}
\|\tau\|_{L^{\infty}_t(B^{\frac{2}{p}}_{p,1})\cap L^{1}_t(B^{2+\frac{2}{p}}_{p,1})}\leq & \|\tau_0\|_{B^{\frac{2}{p}}_{p,1}}
+C\int_0^t\|u\|_{L^2}\|\tau\|_{B^{2+\frac{2}{p}}_{p,1}}
+\|\tau\|_{L^\infty}\|u\|_{B^{1+\frac{2}{p}}_{p,1}}dt
\leq Ce^{e^{Ct}},
\end{align}
where we have used \eqref{x6} and Gronwall inequality.
This complete the proof of Theorem \ref{Th2}.
\par

\end{proof}

{\section { Unstability of the solutions}
\setcounter{equation}{0}
To prove Theorem \ref{Th3}, we firstly give the following local-wellposedness theorem:
\begin{theorem}\label{Unstability1}
 Let $(u_0,\tau_0)\in B^{s+1}_{2,1} (\mathbb{R}^2)\times B^{s}_{2,1} (\mathbb{R}^2)$ with $s\geq1$. Then there exists a lifespan $T(u_0,\tau_0)$ (which is independent on  $\alpha$) such that the system \eqref{gs2222} with $\alpha\geq 0$ has a unique solution $(u,\tau)\in (C_T(B^{s+1}_{2,1} (\mathbb{R}^2)),C_T(B^{s}_{2,1} (\mathbb{R}^2))\cup L^1_T(B^{s+2}_{2,1} (\mathbb{R}^2)))$. Moreover, if there exists an initial sequence satisfying
 $$\lim_{n\rightarrow \infty}\|u^n_{0}-u_0\|_{B^{s+1}_{2,1}}+\|\tau^n_0-\tau_0\|_{B^{s}_{2,1}}=0,$$
then for any $t\in[0,T]$ we have
 $$\lim_{n\rightarrow \infty}\|u^n-u\|_{L^{\infty}_t({B^{s+1}_{2,1}})}+ \|\tau^n-\tau\|_{L^{\infty}_t({B^{s}_{2,1}})}=0.$$
\end{theorem}
\begin{lemma}\label{gj}
Let $(u_0,\tau_0)\in {B}^{s+1}_{2,1}\times{B}^{s}_{2,1}$ be the initial data of (\ref{gs2222}) with $s\geq 1$ and $a\geq 0$. If there exists an initial sequence $(u^j_0,\tau^j_0)\in {B}^{s+1}_{2,1}\times{B}^{s}_{2,1}$ such that $\|u^j_0-u_0\|_{{B}^{s+1}_{2,1}}+\|\tau^j_0-\tau_0\|_{{B}^{s}_{2,1}}\rightarrow 0\ (j\rightarrow\infty)$, then one can construct a lifespan $T^j$ corresponding to $(u^j_0,\tau^j_0)$ such that
$$T^j\rightarrow T,\quad\quad j\rightarrow\infty ,$$
where the lifespan $T$ corresponds to $(u_0,\tau_0)$. This implies that for these initial data $(u^{j}_0,b^{j}_0),~j\in\mathbb{N}\cup \{\infty\}$, there exists a common lifespan $\bar{T}:=\frac{1}{2}T(u_0,b_0)$ , which is independent of $j$ when $j$ is large.
\end{lemma}
One can refer to \cite{weikui1,weikui2} and take the similar operators to obtain the above two theorem. we omit them here.
\par\noindent
\textbf{The proof of Theorem \ref{Th3}:}
\begin{proof}

(1)~To prove
$$\lim_{a\rightarrow 0}\|u^a-u^0\|_{L^{\infty}_T(B^2_{2,1})}+ \|\tau^a-\tau^0\|_{L^{\infty}_T(B^1_{2,1})}=0,$$
where $T$ is the local common lifespan in Theorem \ref{Unstability1} and Lemma \ref{gj}.
Our main idea is to estimate
\begin{align}\label{3.18}
\|u^a-u^0\|_{L^{\infty}_T(B^2_{2,1})}+ \|\tau^a-\tau^0\|_{L^{\infty}_T(B^1_{2,1})}
&\leq \|u^a-u^a_j\|_{L^{\infty}_T(B^2_{2,1})}+ \|\tau^a-\tau^a_j\|_{L^{\infty}_T(B^1_{2,1})}\notag\\
&\quad +\|u^a_j-u^0_j\|_{L^{\infty}_T(B^2_{2,1})}+ \|\tau^a_j-\tau^0_j\|_{L^{\infty}_T(B^1_{2,1})}\notag\\&
+\|u^0_j-u^0\|_{L^{\infty}_T(B^2_{2,1})}+ \|\tau^0_j-\tau^0\|_{L^{\infty}_T(B^1_{2,1})},
\end{align}
where $(u^a_j,\tau^a_j)$ is the local solution for \eqref{gs2222} with the initial data $(u^a_j,\tau^a_j)$ for every $a\geq 0$ ($j\in\mathbb{N}\cup\{\infty\}$), and $(u^a_{\infty},\tau^a_{\infty}):=(u^a,\tau^a)$.

Firstly, we estimate  $\|u^a_j-u^0_j\|_{L^{\infty}_T(B^2_{2,1})}+ \|\tau^a_j-\tau^0_j\|_{L^{\infty}_T(B^1_{2,1})}$ with fix $j$.
Recall the system:
\begin{equation}
\begin{cases}
\partial_tu^a_j+(u^a_j\cdot\nabla)u^a_j+\nabla p^a_j=\mathrm{div}(\tau^a_j),\\
\partial_t\tau^a_j+(u^a_j\cdot\nabla)\tau^a_j-\Delta\tau^a_j+\tau^a_j+\mathrm{Q}(\nabla u^a_j,\tau^a_j)=\alpha\mathrm{D}u^a_j,\\
u^a_j(0,x)=S_ju_0(x),~~~\tau^a_j(0,x)=S_j\tau_0(x).
\end{cases}
\end{equation}
By the local well-posedness in Theorem \ref{Unstability1} and Lemma \ref{gj} , for any $a,j\geq 0$, there exists a local lifespan $T(u_0,\tau_0)$ independent of $a$ and $j$ such that
\begin{align}\label{3.20}
\| u^a_j\|_{L^{\infty}_T(B^3_{2,1})}
+\|\tau^a_j\|_{L^{\infty}_T(B^2_{2,1})}
\leq C(\|S_ju_0\|_{B^3_{2,1}}
+\|S_j\tau_0\|_{L^{\infty}_T(B^2_{2,1})})
\leq C2^j(\|u_0\|_{B^2_{2,1}}
+\|\tau_0\|_{B^1_{2,1}}).
\end{align}

Then, we give the equation of
$(u^a_j-u^0_j,\tau^a_j-\tau^0_j)$:
\begin{align}\label{3.19}
\left\{
\begin{array}{ll}
(u^a_j-u^0_j)_{t}+u^a_j\nabla (u^a_j-u^0_j)+(u^a_j-u^0_j)\nabla u^0_j+\nabla (P^a_j-P^0_j)=div (\tau^a_j-\tau^0_j),\\[1ex]
(\tau^a_j-\tau^0_j)_{t}+(\tau^a_j-\tau^0_j)-\Delta (\tau^a_j-\tau^0_j)+u^a_j\nabla (\tau^a_j-\tau^0_j)+(u^a_j-u^0_j)\nabla\tau^0_j\\[1ex]
+Q(\nabla (u^a_j-u^0_j),~\tau^0_j)+Q(\nabla u^a_j,~(\tau^a_j-\tau^0_j))=\alpha\mathbb{D}u^a_j\\[1ex]
(u^a_j-u^0_j)(0,x)=0,~(\tau^a_j-\tau^0_j)(0,x)=0\\[1ex]
\end{array}
\right.
\end{align}
Applying Lemma \ref{heat}--\ref{priori estimate} and the Gronwall inequality, we obtain
\begin{align}\label{3.20}
\| u^a_j-u^{0}_j\|_{L^{\infty}_T(B^1_{2,1})}+\|\tau^a_j-\tau^{\infty}_j\|_{L^{\infty}_T(B^1_{2,1})}
&\leq 2^jC_{u_0,\tau_0}(\int_0^T\| u^a_j-u^{0}_j\|_{B^2_{2,1}}
\\
&~~~~+\|\tau^a_j-\tau^0_j\|_{B^1_{2,1}}
+a\|u_0\|_{B^2_{2,1}}dt)\notag\\
&\leq C_{u_0,\tau_0}e^{2^jC_{u_0,\tau_0}}a,\notag\\
&\rightarrow 0,~~~a\rightarrow \infty,~for~any~ fixed~ j.
\end{align}

Next, we estimate $\|u^a-u^0\|_{L^{\infty}_T(B^2_{2,1})}+ \|\tau^a-\tau^0\|_{L^{\infty}_T(B^1_{2,1})}$.~
By Theorem \ref{Unstability1}, we see that system \eqref{gs2222} is locally steady for any initial data. Since $\| u_0-S_ju_0\|_{B^2_{2,1}}+
\|\tau_0-S_j\tau_0\|_{B^1_{2,1}}\rightarrow 0,~j\rightarrow \infty$, so we have
\begin{align}\label{3.21}
\| u^a-u^a_j\|_{L^{\infty}_T(B^2_{2,1})}
+\|\tau^a-\tau^a_j\|_{L^{\infty}_T(B^1_{2,1})}
\rightarrow 0,~j\rightarrow \infty,~~for~any~a\geq 0.
\end{align}

Finally, combining \eqref{3.20} and \eqref{3.21}, we deduce that
$$\lim_{a\rightarrow 0}\|u^a-u^0\|_{L^{\infty}_T(B^1_{2,1})}+ \|\tau^a-\tau^0\|_{L^{\infty}_T(B^1_{2,1})}=0.$$
Indeed, for any $\frac{\epsilon}{3}$, by \eqref{3.21}, there exists a constant $N(\epsilon)$ such that when $j\geq N$, we have
$$\|u^a-u^a_j\|_{L^{\infty}_T(B^2_{2,1})}
+\|\tau^a-\tau^a_j\|_{L^{\infty}_T(B^1_{2,1})}\leq \frac{\epsilon}{3},~for~any~a\geq 0.$$
For this $j$, by \eqref{3.20}, there exists a constant $\delta$ ($\delta$ is actually dependent on $\epsilon$) such that when $0\leq a\leq \delta$, one have
$$\|u^a_j-u^0_j\|_{L^{\infty}_T(B^2_{2,1})}
+\|\tau^a_j-\tau^0_j\|_{L^{\infty}_T(B^1_{2,1})}\leq \frac{\epsilon}{3}.$$
Therefore, by \eqref{3.18} we obtain
\begin{align}
\|u^a-u^0\|_{L^{\infty}_T(B^2_{2,1})}+ \|\tau^a-\tau^0\|_{L^{\infty}_T(B^1_{2,1})}
&\leq \|u^a-u^a_j\|_{L^{\infty}_T(B^2_{2,1})}+ \|\tau^a-\tau^a_j\|_{L^{\infty}_T(B^1_{2,1})}\notag\\
&\quad +\|u^a_j-u^0_j\|_{L^{\infty}_T(B^2_{2,1})}+ \|\tau^a_j-\tau^0_j\|_{L^{\infty}_T(B^1_{2,1})}\notag\\
&\quad
+\|u^0_j-u^0\|_{L^{\infty}_T(B^2_{2,1})}+ \|\tau^0_j-\tau^0\|_{L^{\infty}_T(B^1_{2,1})}\notag\\
&\leq \frac{\epsilon}{3}+\frac{\epsilon}{3}+\frac{\epsilon}{3}\leq
\epsilon.
\end{align}

(2) Let  $\epsilon_0=\frac{1}{64(C^4+1)} $ be the fixed small constant in Theorem \ref{Th1} and Theorem \ref{Th2}. Recall the system:
\begin{equation}\label{th4-1}
\begin{cases}
\partial_tu+(u\cdot\nabla)u+\nabla p=\mathrm{div}(\tau),\\
\partial_t\tau+(u\cdot\nabla)\tau-\Delta\tau+\tau+\mathrm{Q}(\nabla u,\tau)=a\mathbb{D}u,~~a\in(0,\epsilon_0],\\
\mathrm{div} u=0,\\
u(x,0)=u_0(x),~~~\tau(0,x)=\tau_0(x),
\end{cases}
\end{equation}
and
\begin{equation}\label{th4-2}
\begin{cases}
\partial_tu+(u\cdot\nabla)u+\nabla p=\mathrm{div}(\tau),\\
\partial_t\tau+(u\cdot\nabla)\tau-\Delta\tau+\tau+\mathrm{Q}(\nabla u,\tau)=0,\\
\mathrm{div} u=0,\\
u(x,0)=u_0(x),~~~\tau(0,x)=\tau_0(x),
\end{cases}
\end{equation}
To prove the instability between \eqref{th4-1} and \eqref{th4-2} as $a\rightarrow 0$, we first give the definition of the global stability:
$$\lim_{a\rightarrow 0}\|u^0-u^a\|_{L^{\infty}_t(B^{2}_{2,1})}
 +\|\tau^0-\tau^a\|_{L^{\infty}_T(B^{1}_{2,1})}=0,~~\forall (u_0,\tau_0)\in \mathbb{A} ~and~\forall t\in [0,\infty),$$
 where $\mathbb{A}:=\{(u_0,\tau_0)\in (B^{2}_{2,1} (\mathbb{R}^2),B^{1}_{2,1} (\mathbb{R}^2))|~\eqref{th4-2}~and~\eqref{th4-1}~ has~a~unique~solution~for~any~fixed~ a\}.$

In order to prove the instability in large time, we should prove that
for any $a\in[0,\epsilon_0]$, there exists a common initial sequence~ $(u_0,\tau_0)(a)\in\mathbb{A}$ and a $T(a)$ such that, when $t\geq T(a)$, we have
$$\|(u^0-u^a)(t)\|_{B^{2}_{2,1}}\geq \frac{\epsilon_0}{8},
$$

Now, let $\phi\in \mathbb{S}^2$ with $\|\phi\|_{L^2}=1$ and $div\phi=0$. Set the initial data $$(u_0,\tau_0)(a)=\frac{\epsilon_0 a}{2}(\phi(ax), a^2\phi(ax)),$$  For any $0< a\leq \epsilon_0 $, we have
$$\|u_0\|_{L^2}=\frac{\epsilon_0}{2},~
\|\tau_0\|_{H^3}\leq C\epsilon_0a^2
~~and~~\|w_0\|_{B^0_{\infty,1}}
\leq\|w_0\|_{H^2}\leq C\epsilon^2_0,~~w_0=curlu_0.$$
These satisfy the conditions in Theorem \ref{Th1} and Theorem \ref{Th2}, which means $(u_0,\tau_0)\in\mathbb{A}$.

On one hand, by Theorem \ref{Th1}, \eqref{th4-1} has an unique global strong solution $(u^a,\tau^a)$ with the initial data $(u_0,\tau_0)(a)$. We also obtain the $L^2$ decay such that:
\begin{align}\label{4.00}
 \|(u^a,\tau^a)(t)\|_{L^2}\leq C(\|u_0\|_{H^1}+\|\tau_0\|_{H^1})[a(1+t)]^{-\frac{1}{2}}.
\end{align}

On the other hand, for the system \eqref{th4-2}, by Theorem \ref{Th2} we also obtain a unique global solution $(u^0,\tau^0)$ with the same initial data $(u_0,\tau_0)(a)$. Although coefficients of the system \eqref{th4-2} are independent of $a$, one can still look for the initial data which is dependent of $a$. Then, using the first equation of \eqref{th4-2}, We have
\begin{align}\label{4.0}
 \|u^0(t)\|^2_{L^2}\geq \|u_0\|^2_{L^2}-\int_0^t\|div\tau\|_{L^2}\|u\|_{L^2}.
 &\geq \|u_0\|^2_{L^{2}}-\|e^{-\frac{s}{2}}\|_{L^2_t}\|e^{\frac{s}{2}}div\tau\|_{L^2_t(L^2)}\|u\|_{L^{\infty}_t(L^2)}\notag\\
 &\geq \frac{\epsilon^2_0}{4}-C\|\tau_0\|_{L^2}(\|u_0\|_{L^2}+\|\tau_0\|_{L^2})\notag\\
 &\geq \frac{\epsilon^2_0}{4}-C\frac{\epsilon^3_0}{2} \notag\\
 &= \frac{\epsilon^2_0}{16},
\end{align}
where we use \eqref{x5} for the third inequality.

Therefore, there exists a $T(a):=\frac{1}{a^2}$ such that when $t\geq T(a)$, we have
 \begin{align}\label{4.1}
 \|u^0(t)-u^\alpha(t)\|_{L^{2}}&\geq \|u^0(t)\|_{L^{2}}-\|u^\alpha(t)\|_{L^{2}}\notag\\
 &\geq \frac{\epsilon_0}{4}
 -C(\|u_0\|_{H^1}+\|\tau_0\|_{H^1})[a(1+t)]^{-\frac{1}{2}}\notag\\
 &\geq \frac{\epsilon_0}{4}
 -C\epsilon_0[a(1+t)]^{-\frac{1}{2}}\notag\\
 &\geq \frac{\epsilon_0}{4}-\frac{\epsilon_0}{8} \notag\\
 &= \frac{\epsilon_0}{8},
\end{align}
where the second inequality is based on \eqref{4.00}. This complete the proof Theorem \ref{Th3}.
\par

\end{proof}

\section*{Acknowledgments}
This work is partially supported by the National Natural Science Foundation of China (Nos. 11801574, 11971485), Natural Science Foundation of Hunan Province (No. 2019JJ50788), Central South University Innovation-Driven Project for Young Scholars (No. 2019CX022) and Fundamental Research Funds for the Central Universities of Central South University, China (Nos. 2020zzts038, 2021zzts0041).

\par\noindent
\bibliographystyle{siam}
\bibliography{ChenYYreference}
\end{document}